\documentclass[12pt]{article}
\usepackage{amsmath}
\usepackage{mathrsfs} 
\usepackage{amsthm}
\usepackage{amsfonts}
\usepackage{amssymb}
\usepackage{latexsym}
\usepackage{mathtools}
\usepackage{yfonts}
\usepackage{tikz} 
\usepackage{esint} 
\usepackage{tikz}
\usepackage{tikz-cd}

\usepackage[matrix,tips,graph,curve]{xy}



\makeatletter 
\@addtoreset{equation}{section}
\makeatother

\theoremstyle{plain}
\newtheorem{theorem}[equation]{Theorem}
\newtheorem{corollary}[equation]{Corollary}
\newtheorem{lemma}[equation]{Lemma}

\theoremstyle{definition}
\newtheorem{definition}[equation]{Definition}

\newtheorem{remark}[equation]{Remark}
\newtheorem{property}[equation]{Property}

\newtheorem{notation}[equation]{Notation}

\newcommand{\IC}{\mathbb{C}}

\newcommand{\IR}{\mathbb{R}}

\newcommand{\IZ}{\mathbb{Z}}
\newtheorem*{definition-non}{definition}

\def\d/{/\mspace{-6.0mu}/}

\begin{document}
\title{ An estimate on energy of min-max Seiberg-Witten Floer generators}
\author{Weifeng Sun}
\date{}
\maketitle 

~\\

In \cite{huchings}, Cristofaro-Gardiner, Hutchings and Ramos proved that embedded contact homology (ECH) capacities can recover the volume of a contact 3-manifod. In particular, a certain sequence of ratios constructed from ECH capacities, indexed by positive integers, was shown to converge to the volume in the index $k \rightarrow +\infty$ limit. There were two main steps in \cite{huchings} to proving this theorem: The first step used estimates for the energy of min-max Seiberg-Witten Floer generators to see that the $k \rightarrow +\infty$ limit of the ratios was a lower bound for the volume. The second step used embedded balls in a certain symplectic four manifold to prove that the $k \rightarrow \infty$ limit of the ratios was an upper bound.\\

Stronger estimates on the energy of min-max Seiberg-Witten Floer generators are derived in this paper that give an effective bound for finite index $k$ on the norm of the difference between the ECH ratio at index $k$ and the volume. This bound implies directly (by taking $k \rightarrow \infty$ ) the theorem in \cite{huchings} that ECH capacities recover volume.\\

Section 1 and Section 2 introduce the notation used in this paper and set the background for the main theorem. Section 3 and Section 4 prove the paper’s main theorem. Section 5 is an addendum that talks about the Seiberg-Witten Floer min-max generators.\\

~\\

~\\

~\\

~\\

~\\

\tableofcontents~\\

~\\

~\\

~\\

~\\

~\\

\section{Notations and prior knowledge}~\\

This section introduces the notations that will be used and reviews some of the
background about the Seiberg-Witten Floer homology.\\

\subsection{Seiberg-Witten equations}~\\

Suppose $(Y, \lambda)$ is a closed, connected, smooth three manifold with a contact form $\lambda$, and is also equipped with a metric satisfying :$|\lambda|=1$, $|d\lambda|=2$, $\textnormal{Vol}(Y)= \dfrac{1}{2}\lambda \wedge d\lambda$. Thus, $*\lambda = \dfrac{1}{2} d\lambda,~~ *d\lambda = 2\lambda. $\\

Choose a $\textnormal{Spin}_{\IC}$ structure on $Y$ with spinor bundles $S$. A connection on $S$ compatible with the metric on $Y$ is uniquely determined by its induced connection on $\det{S}$. Let $A$ denote a connection on $\det{S}$, and $D_A$ denote its Dirac operator on $S$. Moreover, choose a fiducial connection $A_0$ and write $A=A_0+2a$. \\  

This paper will only consider the case when $c_1(\det{S})$ is torsion.\\

Let $\psi$ denote a section of $S$. A pair $(a, \psi)$ is called a ``configuration", typically denoted by $c$. When $\psi= 0$, it is called reducible, otherwise irreducible.\\

\begin{definition}

In \cite{Weinstein}, Taubes considered a perturbed version of Seiberg-Witten equations:\\

\begin{equation}\label{sw1}
    *da=r(\psi ^+\iota \psi -i\lambda)+*d\mu + \mathfrak{T}(a,\psi),
\end{equation}

\begin{equation}\label{sw2}
  2r D_A\psi =2r D_{A_0+2a}\psi=\mathfrak{S}(a,\psi).
\end{equation}\\
Here $\mu$ is a one form, $ \mathfrak{T}, \mathfrak{S}$ are perturbations. Moreover, suppose $P$ is the big Banach space of tame perturbations created in \cite{bible}. ${\mathfrak{T},\mathfrak{S}}$ can be chosen to be the gradient of some $g\in P$ with $\|g\|_P$ bounded.\\

\end{definition} 

\begin{notation}
Let $(SW)_{r,e_{\mu}+g}$ denote the equations (\ref{sw1}), (\ref{sw2}).\\
$(SW)_{r,e_{\mu}}$ and $(SW)_{r}$ means $``g=0"$ and $``g=0, \mu=0"$ versions respectively.\\
Let $N_{r,e_{\mu}+g}$ denote the set of all solutions to $(SW)_{r,e_{\mu}+g}$.\\
\end{notation}

\begin{definition}\label{done}

The book \cite{bible} used $N_{r,m_{\mu}+g}$ (for a generic $g$) to define a ``from" version of Seiberg-Witten Floer homology $\widehat{HM}_{-k}(Y)_{r,e_{\mu}+g}$ (where $-k$ is the degree). It is not necessary to recall the full definition here, but the following is what is needed to know:\\

(1) For any $r$ and $\mu$, $\widehat{HM}_{-k}(Y)_{r,e_{\mu}+g}$ is only defined for a generic $g$. However, for different $(r,\mu, g)$ and $(r',\mu',g')$ where it is defined, there is a canonical isomorphism $T: \widehat{HM}_{-k}(Y)_{r,e_{\mu}+g} \rightarrow \widehat{HM}_{-k} (Y)_{r',e_{\mu'}+g'}$ to identify them, so one can talk about $\widehat{HM}_{-k}(Y)$ without referring to $(r, \mu, g)$. I will discuss more about the isomorphism $T$ later.\\

(2) The generators of the complex used to define $\widehat{HM}_{-k}(Y)_{r,e_{\mu}+g}$ are of two sorts: Each reducible solution contributes infinitely but countably many generators with different degrees which are bounded from above but not from below (I will talk more about this later); Each irreducible solution contributes only one generator with a unique degree, which is also called the degree of the irreducible solution. \\

(3) Fix $k, \mu$ and suppose $\|g\|_P$ is small and bounded (and $g$ is generic), then for $r$ large enough, all generators are contributed by irreducible solutions (see \cite{Weinstein}).\\

(4) Fix $r, \mu, g $ ($g$ generic), for $k$ large enough, all generators are contributed by reducible solutions (this is because there are only finitely many irreducible solutions for fixed $(SW)_{r,e_{\mu}+g}$, see \cite{Weinstein}).\\

(5) This paper only cares about  the mod 2 homologies. And it always regards gauge equivalence configurations as the same thing (for example, a sequence of configurations converges to another means they converge modulo gauge equivalence).\\

\end{definition}

\subsection{Functionals $\mathfrak{a}, cs$ and $E$}~\\

\begin{definition}
Given a configuration $c=(a,\psi)$, its $\mathfrak{a}, cs$ and $E$ are defined as:

$$E=i\int_Y \lambda \wedge (da+*\bar{\omega}_K),~~~~~~cs=-\int_Y a\wedge da,  $$

$$\mathfrak{a}=\dfrac{1}{2}(cs-rE)+r\int_Y<D_{A_0+a}\psi, \psi>. $$

$\bar{\omega}_K$ here is a balanced term, whose definition refers to (2.3) of \cite{Weinstein} and is omitted here. Since $i\int_Y \lambda \wedge *\bar{\omega}_K$ is just a constant, it is not important when doing estimates in this paper.\\

Moreover, when $\mu$ and $g\in P$ are also chosen, one can define:

$$ e_{\mu}= i\int_Y \mu \wedge da,$$

and

$$\mathfrak{a}_{\mu}=\mathfrak{a}+e_{\mu}, ~~~~~~ \mathfrak{a}_{\mu,g}=\mathfrak{a}+e_{\mu}+g.$$

Then $(SW)_{r,e_{\mu}+g}$ is equivalent to the assertion that $\nabla \mathfrak{a}_{\mu,g}(a,\psi)=0$.\\

\end{definition}

\section{Introduction to the main conclusions}~\\

\subsection{Actions on min-max generators}~\\

Fix a homology class from $\widehat{HM}_{-k}(Y)$, denoted by $\{\sigma\}$. In \cite{Weinstein},  the ``min-max" generators were defined in $\widehat{HM}_{-k}(Y)_{r,e_{\mu}}$ for the class $\{\sigma\}$ when  $\mu$ is generic and $r$ is large. \\

However, in this paper, \textbf{the min-max generator for $r\geq 0, \mu$, and $\{\sigma\}$} , denoted as $\hat{c}(r)_{\mu}$, is slightly different from that in \cite{Weinstein}, and thus carries more features. Its definition defers to Section \ref{mm} (lemma \ref{min max generator in section 5}).\\

Now, suppose $\hat{c}(r)_{\mu}$ is given, then there is a number $r_k$ related to it:

\begin{definition}\label{rk}
Suppose $\{\sigma\}$ is fixed and has degree $-k$ with $k$ large, $\mu$ is chosen to be generic, then

$$ r_k=\inf\limits\{s\geq 1 |~ \hat{c}(r)~ \textnormal{is irreducible whenever}~r>s \}. $$~\\
 
\end{definition}

\begin{remark}
(1) In fact, $r_k$ depends on not only $k$, but also $\{\sigma\}$ and $\hat{c}(r)$ (see Section 5 for details). However, $\{\sigma\}$ is assumed to be fixed and $\hat{c}(r)$ is assumed to be chosen a prior, so they are not indicated in the notation $r_k$.\\

(2) $r_k$ is finite. This will be explained in Subsection  \ref{subsection of upper bound}, or see \cite{Weinstein}.\\

(3) When $r>r_k$, $\hat{c}(r)$ must be irreducible. When $r\leq r_k$, $\hat{c}(r)$ can be either reducible or irreducible. However, there exists a nondecreasing sequence $0< s_1\leq s_2\leq \cdots $ with $\lim\limits_{j\rightarrow \infty} s_j=r_k$ such that $\hat{c}(s_j)  ~(j=1,2,3,\cdots$) are all reducible.\\
\end{remark}

Property \ref{min-max} below lists all the features of $\hat{c}(r)_{\mu}$ that will be needed in this paper.\\

\begin{property}\label{min-max}

(1) For any $\mu$, the action $\hat{\mathfrak{a}}(r)=\mathfrak{a}_{r,e_{\mu}}(\hat{c}(r))$  is a continuous function of $r\geq 1$.\\

(2) For a generic $\mu$, when $r>r_k$, $\hat{\mathfrak{a}}(r)$ is continuous and piecewise differentiable. Its differential, $ \dfrac{d\hat{\mathfrak{a}}(r)}{dr}$, is equal to $-\dfrac{1}{2}\hat{E}(r)$, where $\hat{E}(r)=E(\hat{c}(r))$ is only piecewise continuous.\\

(3) Suppose $\mu$ is bounded, $\textnormal{degree}{\{\sigma\}}= -k$ and $k$ is large enough, then $\hat{c}(2)$ is reducible.\\

(4) Although $\hat{c}(r)$ is called min-max generator for convenience, it might not be an actual min-max component of the homology $\{\sigma\}$. In fact, in \cite{Weinstein}, Taubes's min-max generator (denoted as $\hat{c}_T(r)$ here) is an actual min-max component of $\{\sigma\}$. However, his $\hat{c}_T(r)$ is only defined when $r\in U$, where $U$ is an open dense subset of $(r_k, \infty)$. The crucial relationship between $\hat{c}_T(r)$ and $\hat{c}(r)$ is

\begin{equation}\label{buchong}
\lim\limits_{r\in U, r\rightarrow \infty} |E(\hat{c}_T(r))-E(\hat{c}(r))| = 0.
\end{equation}

All the features above will be illustrated in Section \ref{mm}.\\

\end{property}

\begin{remark}

(1) $\hat{c}(r)$ itself is not uniquely determined. But any choice obeying the requirements (1)-(4) of the definition
will suffice.\\

(2)  Even when $r\in U$, $\hat{c}_T(r)$ might be different from $\hat{c}(r)$. ~~(If you see these definitions carefully, $\hat{c}(r)$ and $\hat{c}_T(r)$ are not even solutions to the same Seiberg-Witten equation. This is because Taubes used some extra r-dependent perturbation of Seiberg-Witten equation to define $\hat{c}_T(r)$, see part (d) of Section 3 of \cite{Weinstein}.) The reason that $\hat{c}(r)$ is still useful is the identity (\ref{buchong}).\\

(3) Although $\hat{\mathfrak{a}}(r)$ is a continuous function of $r$, $\hat{c}(r)$ might not be continuous. In fact, $\hat{c}(r)$ is only piecewise continuous when $\mu$ is generic and $r>r_k$ (see Subsection \ref{piecewise continuity} for details).\\

(4) Usually $\{\sigma\}\in \widehat{HM}_{-k}(Y)$ is chosen a prior and then not mentioned subsequently.\\

\end{remark}

\subsection{The key estimate on the energy}~\\

\begin{remark}
In this paper, $C$ always means some big enough positive constant which is independent with $r,k$, but can have different values in different formulas.  Plus, $O(\cdots)$ means absolutely smaller than $C\cdot (\cdots)$.\\
\end{remark}

Granted the definitions in 2.1, here is the main theorem:\\

\begin{theorem}\label{main theorem}\label{pre rk}
Suppose that $k$ is a large integer, $\textnormal{degree} \{\sigma\}=-k$; and that $\mu$ is chosen to be generic and have small norm. There exists $r_k \geq 2$ (see definition \ref{rk}) such that if $r>r_k$, then

$$|\dfrac{\hat{E}(r)^2}{8\pi^2k}-{Vol(Y)}| = O(k^{-\frac{1}{126}}). $$~\\

\end{theorem}

The proof of Theorem \ref{main theorem} is in Section \ref{key}.\\

\subsection{Application on ECH capacities}~\\

Recall in \cite{huchings} that $\lim\limits_{r\in U, r\rightarrow+\infty}E(\hat{c}_T(r))=2\pi c_{\sigma}(Y,\lambda)$, $c_{\sigma}(Y,\lambda)$ is the ECH capacity. Notice, here $\{\sigma\}$ should be understood as a Seiberg-Witten Floer cohomology class, and $\hat{c}_T(r)$ should be understood as the min-max generator of Seiberg-Witten Floer cohomology. However, the estimates in paper are still valid for Seiberg-Witten Floer cohomology, so it is safe to use the same notation as for Seiberg-Witten Floer homology elsewhere. So, together with formula (\ref{buchong}), here is a corollary:\\

\begin{corollary}\label{big corollary}~~~~~~
$|\dfrac{c_{\sigma}(Y,\lambda)^2}{2k} -\textnormal{Vol(Y)}| = O(k^{-\frac{1}{126}}).$\\

Moreover, suppose $\{\sigma_m\} (m=0,1,2,3,\cdots)$ is a certain sequence of ECH classes, whose $m$th term has a degree $k_m\in \IZ$ and $\lim\limits_{m\rightarrow +\infty}k_m = +\infty$. (In fact, ``certain sequence" here means $c_1(S)$ is torsion, so that $k_m$ can be defined as an integer, though the way to define it is not unique. This is the case correspondence to the Seiberg-Witten cohomology discussed in this paper. See \cite{huchings} for details.) Then,
$$\lim\limits_{m\rightarrow +\infty}\dfrac{c_{\sigma_m}(Y,\lambda)^2}{2k_m}=\textnormal{Vol(Y)}.$$\\
\end{corollary}

Notice, this gives a purely analytic proof of ``ECH capacities recover volume theorem" (see \cite{huchings}) with an estimate on the speed.\\

\begin{remark}

There might be a potential further application of this:\\
In \cite{dense}, Irie used ``ECH capacities recover the volume" theory to prove that on compact 3-manifold, Reeb orbits are dense for a generic contact form. As a corollary, on a compact 2-manifold, closed geodesics are dense for a generic metric. Corollary \ref{big corollary} might carry some hints to a quantitative estimate of the above theorem. (For example, as I suppose, on a 2-dimensional compact manifold with a given metric $g$, there seems to exist a metric $g_{\epsilon}$ such that $\|g_{\epsilon}-g\|\leq \epsilon$ and $g_{\epsilon}$ has a closed geodesic with length at most $C\epsilon^{-\delta}$, where $C$ is independent with $\epsilon$, and $\delta$ is independent with everything.)\\

\end{remark}

Now, let's start to do analysis!\\

\section{Some preliminary estimates}~\\

\subsection{Estimates from Taubes}\label{estimates from taubes}~\\

From \cite{Weinstein} and \cite{huchings}, many inequalities are obtained to be used. They are stated in the following:\\

\begin{lemma}\label{pre}
(1) Suppose $(a,\psi)$ is a solution to $(SW)_{r,e_{\mu}+g}$ ($g$ is generic, $\mu$ and  $g $ are bounded), then

\begin{equation}
E\leq r \textnormal{Vol}(Y)+C.
\end{equation}~\\

(2) If $(a,\psi)$ in (1) is irreducible and suppose its $E,r$ has a positive lower bound, then

\begin{equation}\label{sfe}
|cs+2e_{\mu}+2g-4\pi^2k|\leq C r^{\frac{31}{16}},
\end{equation}

\begin{equation}\label{cse}
|cs+2e_{\mu}+2g|\leq Cr^{\frac{2}{3}} E^{\frac{4}{3}}.
\end{equation}~\\

(3) If $(a,\psi)$ in (1) is reducible, and $r$ has a positive lower bound, then

\begin{equation}
cs+2e_{\mu}+2g= \dfrac{1}{2} r^2 \textnormal{Vol}(Y) + O(r),
\end{equation}

\begin{equation}\label{E estimate}
E= r \textnormal{Vol}(Y) + O(1),
\end{equation}

\begin{equation}\label{estimate of action of red}
\mathfrak{a}= -\dfrac{1}{4} r^2 \textnormal{Vol}(Y) + O(r).
\end{equation}\\

\end{lemma}

Moreover, a corollary can be derived from the above estimates which will be used later:\\

\begin{corollary}
$c$ is a irreducible solution to $(SW)_{r,e_{\mu}+g}$ with $r\geq 1$ and $\mu, g$ bounded, then

\begin{equation}\label{estimate of action on irr}
\mathfrak{a}_{r,e_{\mu}+g}(c) > 2\pi^2k-\dfrac{1}{2}r^2\textnormal{Vol(Y)}-Cr^{\frac{31}{16}}.\\
\end{equation}

\end{corollary}

\begin{proof}
This is a corollary directly from (\ref{sfe}),(\ref{E estimate}) and the fact that, for a solution,~~
$\mathfrak{a}_{r,e_{\mu}+g}(c)=\dfrac{1}{2}(cs-rE)+e_{\mu}+g.$\\
\end{proof}

\subsection{A lower bound of $r_k$}~\\

\begin{lemma}
When $\mu$ is bounded, and when $k$ is large, then $r_k\geq 2$.\\
\end{lemma}

\begin{proof}
This is just because of property (3) in property \ref{min-max}.\\
\end{proof}

\begin{theorem} Here is an estimate on $r_k$:\label{estimate on rk}
$$r_k^2 \geq  \dfrac{8\pi ^2 k}{ \textnormal{Vol(Y)}}-C k^{\frac{32}{33}}.$$~\\
\end{theorem}

\begin{proof}
If $r^2< \dfrac{8\pi ^2 k}{ \textnormal{Vol(Y)}}-Ck^{\frac{32}{33}}$, let $c_{irr}$, $c_{red}$ be any irreducible and reducible solutions to $(SW)_{r,e_{\mu}}$ respectively.\\

From (\ref{estimate of action on irr}) one gets
$$\mathfrak{a}_{r,e_{\mu}}(c_{irr})> 2\pi^2k-\dfrac{r^2}{2} \textnormal{Vol(Y)}-Cr^{\frac{31}{16}} >  -\dfrac{r^2}{4} \textnormal{Vol(Y)}+ Ck^{\frac{33}{34}}+O(r^{\frac{31}{16}})\geq - \dfrac{r^2}{4} \textnormal{Vol(Y)}+ Ck^{\frac{33}{34}}.$$

The last step is because $k=O(r^2)$, so $r^{\frac{31}{16}}=o(k^{\frac{33}{34}}).$\\

However, from (\ref{estimate of action of red}) one gets 
$$\mathfrak{a}_{r,e_{\mu}}(c_{red})< -\dfrac{1}{4}r^2 \textnormal{Vol(Y)}+ Cr <  - \dfrac{r^2}{4} \textnormal{Vol(Y)}+ Ck^{\frac{33}{34}} < \mathfrak{a}_{r,e_{\mu}}(c_{irr}).$$

Notice $\hat{\mathfrak{a}}(r)$ is continuous w.r.t. r, so when $1< r^2 < \dfrac{8\pi^2 k}{ \textnormal{Vol(Y)}}-Ck^{\frac{32}{33}}$, $\hat{c}(r)$ cannot shift between reducible and irreducible (since there is always a positive gap between their actions). \\

Since $\hat{c}(2)$ is reducible, so $\hat{c}(r)$ must be reducible as long as\\
$r^2< \dfrac{8\pi^2 k}{ \textnormal{Vol(Y)}}-Ck^{\frac{32}{33}}$ , which implies $r_k^2 \geq \dfrac{8\pi^2 k}{ \textnormal{Vol(Y)}}-Ck^{\frac{32}{33}}.$\\

\end{proof}

\begin{remark}
Before moving to the upper bound of $r_k$, I want to introduce a fake proof of theorem \ref{estimate on rk} (which confused me a lot before), which is incorrect but carries some hints and clarifications on what to expect:\\

Since $r_k$ is the borderline between where $\hat{c}(r)$ to be irreducible and reducible, it should satisfy all of (1) (2) (3) in lemma \ref{pre}, which implies
$$cs(\hat{c}(r_k))+2e_{\mu}(\hat{c}(r_k)) = \dfrac{1}{2} r_k^2 \textnormal{Vol(Y)}+O(r_k)=8\pi^2k+O(r_k^{\frac{31}{16}}).$$

Thus $r_k^2=\dfrac{8\pi^2 k}{\textnormal{Vol(Y)}}+O(k^{\frac{31}{32}})$.\\

This above argument is invalid because $cs(\hat{c}(r_k))+2e_{\mu}(\hat{c}(r_k)) $ is not continuous in general, and also because the spectral flow estimate,  i.e., (2) of the lemma \ref{pre} (also see \cite{Weinstein}) ) is invalid for generators contributed from reducible solutions, even near the borderline. However, one can still say something about the degree of reducible generators from the spectral flow,  which will imply an upper bound of $r_k$.\\

\end{remark}

\subsection{An upper bound of $r_k$}\label{subsection of upper bound}~\\

\begin{theorem}\label{upper bound of rk}

Suppose that $a+g$ is used to replace $a$ for some generic, small normed $g\in P$.  If $a+g$ is used to defined the Seiberg-Witten equations (which will henceforth be assumed), then
$$r_k^2 \leq \dfrac{8\pi^2 k}{ \textnormal{Vol(Y)}}+C k^{\frac{31}{32}}.$$~\\

\end{theorem}

\begin{proof}

From \cite{Weinstein} one knows, each reducible generator of the ``from" version of the Seiberg-Witten Floer complex corresponds to an eigenvector of the Dirac operator $D_{A-ir\lambda+2\mu}$ with negative eigenvalue. The degree of such a generator differs by a constant (independent of the eigenvector, eigenvalue and r) from -2 times the sum of two numbers, $\mathfrak{X}$ and $\mathfrak{Y}$. These are defined as follows: The number $\mathfrak{X}$ is the number of negative eigenvalues above the eigenvalue of the given eigenvector. Meanwhile, the number $\mathfrak{Y}$ is the spectral flow for the family $D_{A-isλ+2μ}, s\in[0,r]$.

(The reason for the use here of a generic $g$ to perturb $\mathfrak{a}$ is that, the above argument requires the Dirac operator to have spectrum with multiplicity 1 for each eigenvalue, see \cite{Weinstein} for details.)\\

Thus, the degree of a reducible generator (when $r\geq 1$) is

$$-k = -2 \mathfrak{X} -2 \mathfrak{Y}+C \leq -2 \mathfrak{Y}+C = -\dfrac{1}{8\pi^2}r^2\textnormal{Vol(Y)}+O(r^{\frac{31}{16}}). $$

Thus $-k\leq - \dfrac{1}{8\pi^2}r^2\textnormal{Vol(Y)}+Cr^{\frac{31}{16}},$ whenever $\hat{c}(r)$ is reducible.\\

Thus $\dfrac{1}{8\pi^2}r_k^2\textnormal{Vol(Y)} \leq k +Cr^{\frac{31}{16}}$, which implies theorem \ref{upper bound of rk}.\\
\end{proof}

Combine theorem \ref{estimate on rk} and theorem \ref{upper bound of rk} together, here is the final conclusion about $r_k$:\\

\begin{theorem}\label{final estimate of rk}
$$r_k^2=\dfrac{8\pi^2 k}{ \textnormal{Vol(Y)}}+O( k^{\frac{32}{33}}).$$

And also,

$$ \hat{\mathfrak{a}}(r_k)=-\dfrac{1}{4}r_k^2\textnormal{Vol(Y)}+O(r_k)=-2\pi^2k+O(k^{\frac{32}{33}}).$$

\end{theorem}

The last step is because $r_k=O(k^{\frac{1}{2}})\leq O(k^{\frac{32}{33}}).$\\

\section{The crucial estimate}\label{key}~\\

In this section, the goal is to prove theorem \ref{main theorem}.\\

\subsection{Differential equations} ~\\

\begin{lemma}
Suppose $\mu$ is generic and bounded.\\

Let $y_1=\dfrac{\hat{cs}(r)-4\pi^2 k+2\hat{e}_\mu(r)}{r} $ and $y_2=\dfrac{\hat{cs}(r)+2\hat{e}_\mu(r)}{r}, $\\

When $r > r_k$, the functions $E-y_1$ and $E-y_2$ are continuous, piecewise differentiable; and
where they are differentiable, they satisfy the following equation:
\begin{equation}\label{equation}
\dfrac{d(\hat{E}-y_i)}{dr}=\dfrac{y_i}{r}, ~~~i=1,2.
\end{equation}\\

\end{lemma}

\begin{proof}
From property (2) of property \ref{min-max}, one knows 
\begin{equation}\label{4.3}
\dfrac{d \hat{\mathfrak{a}}(r)}{dr}=-\dfrac{1}{2}\hat{E}(r).
\end{equation}
Also notice $y_1=\dfrac{-2\hat{a}-4\pi^2k}{r}+\hat{E}$, $y_2=-\dfrac{2\hat{a}}{r}+\hat{E}$ by definition. Differentiate these formulas using the formula in (\ref{4.3}) for derivatives of $\hat{\mathfrak{a}}$, and one gets (\ref{equation}).\\
\end{proof}

\begin{lemma}
Here are the estimates on initial values:
\begin{equation}\label{I1}
I_1:=(\hat{E}-y_1)(r_k)=r_k\textnormal{Vol(Y)}+O(r_k^{\frac{31}{33}}),
\end{equation}
\begin{equation}\label{I2}
I_2:=(\hat{E}-y_2)(r_k)=\dfrac{1}{2}r_k\textnormal{Vol(Y)}+O(1).
\end{equation}\\
\end{lemma}

\begin{proof}
From theorem \ref{final estimate of rk}, one gets
$$ (\hat{E}-y_2)(r_k)=\dfrac{-2\hat{\mathfrak{a}}(r_k)}{r_k}=\dfrac{1}{2}r_k\textnormal{Vol(Y)}+O(1).$$
Also,
$$ (\hat{E}-y_1)(r_k)=\dfrac{-2\hat{\mathfrak{a}}(r_k)-4\pi^2k}{r_k}=r_k\textnormal{Vol(Y)}+O(1)+\dfrac{O(k^{\frac{32}{33}})}{r_k}=r_k\textnormal{Vol(Y)}+O(r_k^{\frac{31}{33}}).$$
The last step is because $k=O(r_k^2).$\\
\end{proof}

There is one more estimate to exhibit before moving on:\\
\begin{lemma}
For any $r\geq r_k,$
\begin{equation}\label{y1 estimate}
|y_1|\leq C r^{\frac{15}{16}},
\end{equation}
\begin{equation}\label{y2 estimate}
|y_2|\leq C r^{-\frac{1}{3}}\hat{E}(r)^{\frac{4}{3}}.
\end{equation}\\
\end{lemma}

\begin{proof}
(\ref{y1 estimate}) is directly from the definition of $y_1$ and  (\ref{sfe}).\\

(\ref{y2 estimate}) is from the definition of $y_2$ and (\ref{cse}) (notice $\hat{E}$ is bounded from below is because $\{\sigma\}$ is nontrivial, see \cite{Weinstein} for a similar argument).\\

\end{proof}

\subsection{Integrals and asymptotic comparison estimates}~\\

Now let's do the integration estimates.\\

\begin{lemma}
Suppose $r\geq r_k$,
\begin{equation}\label{411}
\hat{E}(r) = r_k\textnormal{Vol(Y)}+ O(r^{\frac{31}{33}}).
\end{equation}

\end{lemma}

\begin{proof}
Plug (\ref{y1 estimate}) into (\ref{equation}), one gets

$$|(\hat{E}(r)-y_1(r))-I_1)|\leq \int^{r}_{r_k} C s^{-\frac{1}{16}} ds= O(r^{\frac{15}{16}}), $$

combining with (\ref{I1}) and (\ref{y1 estimate}) again one gets

$$\hat{E}(r)=r_k\textnormal{Vol(Y)} + O(r_k^{\frac{31}{33}})+O(r^{\frac{15}{16}}) = r_k\textnormal{Vol(Y)}+ O(r^{\frac{31}{33}}).$$

The last step is because $r\geq r_k$ and $\dfrac{31}{33}>\dfrac{15}{16}$.\\

\end{proof}

\begin{lemma}\label{diedai}
Suppose
\begin{equation}\label{5}
\hat{E}(r)=r_k \textnormal{Vol}(Y)+O(r^\delta+r_k^\epsilon)
\end{equation}

with $\delta, \epsilon$ constrained as follows: First $0 < \delta < 1$, and  $\delta \neq \dfrac{1}{4})$. Second,   $\dfrac{4 \delta}{3 \delta +1} \leq \epsilon < 1 $. Then $\hat{E}(r)=r_k$ obeys the stronger bound:

$$ \hat{E}(r)=r_k \textnormal{Vol}(Y)+O(r^{\frac{4}{3}\delta-\frac{1}{3}}+r_k^\epsilon).$$~\\

\end{lemma}

\begin{proof}
Plug (\ref{5}) into (\ref{y2 estimate}), one gets

\begin{equation} \label{6}
|y_2|\leq C r^{-\frac{1}{3}}(r_k \textnormal{Vol}(Y)+O(r^\delta)+O(r_k^\epsilon))^{\frac{4}{3}} = O(r^{-\frac{1}{3}}r_k^{\frac{4}{3}}+r^{\frac{4}{3}\delta-\frac{1}{3}}).
\end{equation}

Choose $r_0\geq r_k$,~ when $r\geq r_0$, plug the above inequality into (\ref{equation}), one gets

\begin{equation} \label{eq1}
\begin{split}
&~~~~|(\hat{E}(r)-y_2(r))-(\hat{E}(r_0)-y_2(r_0))|\\
& \leq C \int^{r}_{r_0}  s^{-\frac{4}{3}}r_k^{\frac{4}{3}}+s^{\frac{4}{3}\delta-\frac{4}{3}} ds\\
& =O(r^{-\frac{1}{3}}r_k^{\frac{4}{3}}+r_0^{-\frac{1}{3}}r_k^{\frac{4}{3}}+r^{\frac{4}{3}\delta-\frac{1}{3}}+r_0^{\frac{4}{3}\delta-\frac{1}{3}}).
\end{split}
\end{equation}

Thus, use (\ref{6}) again to $y_2(r)$ and $y_2(r_0)$ on the left hand side above, one has

\begin{equation}
\begin{split}
&~~~~|\hat{E}(r)-\hat{E}(r_0)|\\
&=O(r^{-\frac{1}{3}}r_k^{\frac{4}{3}}+r_0^{-\frac{1}{3}}r_k^{\frac{4}{3}}+r^{\frac{4}{3}\delta-\frac{1}{3}}+r_0^{\frac{4}{3}\delta-\frac{1}{3}}+r^{\frac{4}{3}\delta-\frac{1}{3}}+r_0^{\frac{4}{3}\delta-\frac{1}{3}})\\
&=O(r^{-\frac{1}{3}}r_k^{\frac{4}{3}}+r_0^{-\frac{1}{3}}r_k^{\frac{4}{3}}+r^{\frac{4}{3}\delta-\frac{1}{3}}+r_0^{\frac{4}{3}\delta-\frac{1}{3}}).
\end{split}
\end{equation}

Remember $r\geq r_0\geq r_k \geq 2$, so whether or not $\dfrac{4}{3}\delta-\dfrac{1}{3}$ is positive, one always has $r_0^{\frac{4}{3}\delta-\frac{1}{3}}=O(r^{\frac{4}{3}\delta-\frac{1}{3}}+r_k^{\epsilon})$, and $r^{-\frac{1}{3}}r_k^{\frac{4}{3}} \leq r_0^{-\frac{1}{3}}r_k^{\frac{4}{3}} $.\\

So together with (\ref{5}) one gets :\\

$\hat{E}(r)=r_k\textnormal{Vol}(Y)+O(r_0^\delta+r_k^\epsilon+r_0^{-\frac{1}{3}}r_k^{\frac{4}{3}}+r^{\frac{4}{3}\delta-\frac{1}{3}}).$\\

The is also true when $r_k \leq r\leq r_0$ directly by (\ref{5}).\\

So by choosing $r_0=r_k^\frac{4}{3 \delta +1} $ so that $r_0^{\delta} = r_0^{-\frac{1}{3}}r_k^{\frac{4}{3}}=r_k^{\frac{4 \delta}{3 \delta +1}} = O(r_k^ \epsilon),$\\

one has $\hat{E}(r)=r_k\textnormal{Vol}(Y)+O(r_k^\epsilon+r^{\frac{4}{3}\delta-\frac{1}{3}}).$\\

\end{proof}

\begin{lemma}\label{niubi}
Same condition as lemma \ref{diedai} , but the result is

$$ \hat{E}(r)=r_k \textnormal{Vol}(Y)+O(r_k^\epsilon).$$\\
\end{lemma}

\begin{proof}
Starting with any $\delta$, iterating lemma \ref{diedai}, by replacing $\delta$ with\\

$\dfrac{4}{3}\delta-\dfrac{1}{3}=1-\dfrac{4}{3}(1-\delta)$ finite many times, and increase a little bit if it touches $\dfrac{1}{4}$, until it is below 0, so the corresponding term can be bounded, and can be absorbed into $O(r_k^{\epsilon})$. ~Finally one can get,

$$\hat{E}(r)=r_k \textnormal{Vol}(Y)+O(r_k^\epsilon).$$

\end{proof}

\subsection{Proof of theorem \ref{main theorem}}~\\

An appeal to lemma \ref{niubi} can be made starting from $\delta=\dfrac{31}{33}, \epsilon = \dfrac{4 \delta}{3 \delta +1} = \dfrac{62}{63}$ because of (\ref{411}). This appeal leads to the bound:\\

$$\hat{E}(r)=r_k \textnormal{Vol}(Y)+O(r_k^\frac{62}{63}).$$

So

$$\hat{E}(r)^2=r_k^2 \textnormal{Vol}(Y)^2+O(r_k^\frac{125}{63}). $$

So use theorem \ref{final estimate of rk} again,

$$\hat{E}(r)^2- 8 \pi^2 k \textnormal{Vol}(Y) = O(r_k^\frac{125}{63}+k^{\frac{32}{33}})= O(k^{\frac{125}{126}}).$$

The last step is because $r_k=O(k^{\frac{1}{2}})$, so $O(r_k^\frac{125}{63}+k^{\frac{32}{33}})=O(k^{\frac{125}{126}}+k^{\frac{32}{33}})=O(k^{\frac{125}{126}}).$\\

Finally, from above, one gets

$$|\dfrac{\hat{E}(r)^2}{8\pi^2 k}-\textnormal{Vol}(Y)|=O( k^{-\frac{1}{126}}).$$~\\

\section{Existence of min-max generators}\label{mm}~\\

This section gives the construction of min-max generators.\\

\subsection{Construction of $\hat{\mathfrak{a}}(r)$ for any $\mu$}~\\

\begin{definition}\label{definition of a(r)}
Fix $r$ and $\mu$, for any integer $m>1$, choose $g_m\in P$ with $\|g_m \|_P < \dfrac{1}{m} $ and generic (so $\widehat{HM}_k(Y)_{r,e_{\mu}+g_m}$ is well-defined). Let

$$\hat{\mathfrak{a}}(r)_{e_{\mu}+g_m}=\min\{\max\{ \mathfrak{a}_{r,e_{\mu}+g_m}(c)| ~c ~\textnormal{is a generator of } \sigma \}  | ~ \sigma~\textnormal{is a representative of} \{ \sigma \} \}.$$

Furthermore, let $\hat{\mathfrak{a}}(r)=\lim\limits_{m\rightarrow +\infty} \hat{\mathfrak{a}}(r)_{e_{\mu}+g_m}$, then one gets:\\
\end{definition}

\begin{lemma}\label{independence}
$\hat{\mathfrak{a}}(r)$  doesn't depend on  $g_m$.\\
\end{lemma}

\begin{proof}

Suppose there are two different ways of choosing $g_m$, denoted as $g_m$ and $g'_m$ separately. Connect them via a generic path $g(s)\in P, ~-\infty<s< +\infty $ which is defined so that $g(s) = g_m +e_{\mu}$ where $s < -1$, and $g(s) = g'_m +e_{\mu}$ where $s > +1$. The path can also be chosen to obey the bound $\|\dfrac{dg}{ds} \| \leq \dfrac{4}{m}$. (The notation here uses $\| \cdot \|$ to denote the P norm defined in \cite{bible}.)\\

Consider the SW trajectories on $ Y \times \IR$ using perturbation $g(s)$. The corresponding instantons on $Y\times \IR$ give an isomorphism $T: \widehat{HM}_{-k}(SW)_{r,g_m+e_{\mu}} \rightarrow \widehat{HM}_{-k}(SW)_{r,g'_m+e_{\mu}}$.\\

To be precise, $T$ is the map $\hat{m}$ defined in definition 25.3.4 of the book \cite{bible}, evaluating at the cohomology class ``1" of the blown-up configuration space of $Y \times \IR$. The above $T$ is a prior only an homomorphism from $ \widehat{HM}_{\bullet}(SW)_{r,g_m+e_{\mu}} $ to $ \widehat{HM}_{\bullet}(SW)_{r,g'_m+e_{\mu}}$. (See theorem 23.1.5 and its corollary in the book \cite{bible}.) Here, $HM_{\bullet}$ stands for the negative completion of the homology, in the sense of definition 3.1.3 of the book \cite{bible}. (This notation is not important in this paper.) \\

However, in the special case as above, $T$ is an isomorphism and keeps the degree $-k$. This is because here $c_1(S)$ is torsion and the perturbation is balanced, and the cobordism $Y \times \IR$ is a cylinder. For a generic $g(s)$, the above $T$ counts the instantons on $Y \times \IR$ in four different ways (see definition 25.3.3 of the book \cite{bible}, where $T$ has four components which form a $2\times 2$ matrix. The four components are correspondence to : (1) irreducible to irreducible, (2)irreducible to reducible, (3)reducible to irreducible, (4) reducible to reducible respectively). Carefully checking them, one finds that in each component, $T$ only counts the (possibly broken) instantons on $Y \times \IR$ which connects elements in  $ \widehat{HM}_{\bullet}(SW)_{r,g_m+e_{\mu}} $ to $ \widehat{HM}_{\bullet}(SW)_{r,g'_m+e_{\mu}}$ with the the same degree. \\

Now let $c, c'$ be solutions to $(SW)_{r,g_m+e_{\mu}}, (SW)_{r,g'_m+e_{\mu}}$ both of degree -k, and with $c'$ being a component of $Tc$. Then, there is at least one instanton trajectory (or possibly a broken one)  connecting $c$ to $c'$. By instanton, I mean a family of configurations parametrized by the coordinate $s$ for $\IR$ obeying the following conditions:  First, the $s \rightarrow -\infty$ limit should be c and the $s \rightarrow \infty$ limit should be gauge equivalence with $c´$  (still denoted as $c'$).  Second, the s-dependent family of configuration should obey the equation:\\

$$\dfrac{d}{ds} c(s) = -\nabla (\textfrak{a}_r+g(s)). $$

Although the definition of $T$ (see $\hat{m}$ in the definition 25.3.4 of the book \cite{bible}) used the blown-up configuration space, the instanton trajectory used here is only its projection to the configuration space without blown up. Granted above, then

\begin{equation}
\begin{split}
\mathfrak{a}_{r,g'_m+e_{\mu}}(c')-\mathfrak{a}_{r, g_m+e_{\mu}}(c)&= \int_{-\infty}^{+\infty}  \frac{d}{ds}(\mathfrak{a}_r(c(s))+g(s)(c(s)))ds \\
&= \int_{-\infty}^{+\infty} (\nabla(\mathfrak{a}_r+g(s))\cdot \dfrac{dc(s)}{ds}+\dfrac{dg(s)}{ds}(c(s)) )ds\\
&= \int_{-\infty}^{+\infty} (-\|\nabla(\mathfrak{a}_r+g(s))\|^2+\dfrac{dg(s)}{ds}(c(s)) ))ds\\
&\leq  \int_{-\infty}^{+\infty} \dfrac{dg(s)}{ds}(c(s))ds \\
&\leq  \int_{-1}^{1} \dfrac{4}{m}\|c(s)\|ds \leq \dfrac{C_r}{m}.
\end{split}
\end{equation}

Here $C_r$ is some constant independent with $m$, $g_m$ and $g'_m$.\\

Let's continue the proof, suppose $\hat{\sigma}_{r,g_m+e_{\mu}}$ is a representative of $\{\sigma\}$ and $\hat{c}_{r, g_m+e_{\mu}}$ is a component of $\hat{\sigma}_{r,g_m+e_{\mu}}$  which achieves the min-max of action, i.e.,  

$$\mathfrak{a}_{r,g_m+e_{\mu}}(\hat{c}_{r,g_m+e_{\mu}}) = \hat{\mathfrak{a}}(r)_{g_m+e_{\mu}}.$$

Let $c$ be any component of $\hat{\sigma}_{r,g_m+e_{\mu}}$, then by definition, $$\mathfrak{a}_{r,g_m+e_{\mu}}(c)\leq \mathfrak{a}_{r,g_m+e_{\mu}}(\hat{c}_{r,g_m+e_{\mu}}).$$

Let $c'$ be any component of $Tc$, one gets , by the above lemma,

$$\mathfrak{a}_{r,g'_m+e_{\mu}}(c')\leq \mathfrak{a}_{r,g_m+e_{\mu}}(c) +\dfrac{C_r}{m}\leq \mathfrak{a}_{r,g_m+e_{\mu}}(\hat{c}_{r,g_m+e_{\mu}})+\dfrac{C_r}{m}.$$

Since the above is true for any component of $T\hat{\sigma}_{r,g_m+e_{\mu}}$, which is a representative of $\{\sigma\}$ in the $g_m'$ version of SW homology, one gets:

$$\hat{\mathfrak{a}}(r)_{g'_m+e_{\mu}}\leq \hat{\mathfrak{a}}(r)_{g_m+e_{\mu}} +\dfrac{C_r}{m}. $$

Similarly,

$$\hat{\mathfrak{a}}(r)_{g_m+e_{\mu}}\leq \hat{\mathfrak{a}}(r)_{g'_m+e_{\mu}} +\dfrac{C_r} {m}.$$ 

So $| \hat{\mathfrak{a}}(r)_{g'_m+e_{\mu}}-\hat{\mathfrak{a}}(r)_{g_m+e_{\mu}} | \leq \dfrac{C_r}{m}$. Let $m\rightarrow +\infty$, it implies lemma \ref{independence}.\\
\end{proof}

\subsection{Continuity of $\hat{\mathfrak{a}}(r)$}~\\

\begin{theorem}\label{5.4}
The $\hat{\mathfrak{a}}(r)$ defined just now is continuous along r.\\
\end{theorem}

\begin{proof}
Fix $r_0$, suppose $g_m$ is chosen as in last Subsection for $r_0$. Notice\\

$(SW)_{r_0+\epsilon, e_{\mu}+e_{\frac{1}{2}\epsilon \lambda}+g_m}$ is the same equation as $(SW)_{r_0, e_{\mu}+g_m}$ for any $\epsilon \in \IR$, and when $|\epsilon|$ is small (say, when $0\leq |\epsilon|<\delta(r_0,m)$), $\|g_m+e_{\frac{1}{2}\epsilon\lambda}\|_P<\dfrac{1}{m}$ still holds true. Thus $g_m+e_{\frac{1}{2}\epsilon\lambda}$ can also play the role of ``$g_m$" with $r=r_0+\epsilon$ when $\epsilon$ is small, and they have the same action on min-max generators. Thus, as long as $|\epsilon|<\delta(r_0,m)$ , $|\hat{\mathfrak{a}}(r_0)_{g_m+e_{\mu}}-\hat{\mathfrak{a}}(r_0+\epsilon)| \leq \dfrac{C_{r_0+\epsilon}}{m}$.   Moreover, remember $|\hat{\mathfrak{a}}(r_0)_{g_m+e_{\mu}}-\hat{\mathfrak{a}}(r_0)|  \leq \dfrac{C_{r_0}}{m}$, 
thus $|\hat{\mathfrak{a}}(r_0)-\hat{\mathfrak{a}}(r_0+\epsilon)|\leq \dfrac{C_{r_0}+C_{r_0+\epsilon}}{m}$. \\
Notice $C_{r}$ is bounded nearby $r_0$ (for $|\epsilon|<\delta(r_0,m)$), so $m \rightarrow \infty$ implies $\dfrac{C_{r_0}+C_{r_0+\epsilon}}{m}$ can be arbitrarily small, which implies $\hat{\mathfrak{a}}$ is continuous.\\

\end{proof}

It is always possible to construct $\hat{c}(r)$ in $(SW)_{r,e_{\mu}}$, if it is only required to have an action equal to $\hat{\mathfrak{a}}(r)$. This is the following lemma:\\

\begin{lemma}\label{min max generator in section 5}
For each $r$, one can choose a solution of $(SW)_{r,e_{\mu}}$, denoted by $\hat{c}(r)_{\mu}$(or $\hat{c}(r)$, $\hat{c}$ for short) , such that $\mathfrak{a}_{r,e_{\mu}}(\hat{c}(r)) = \hat{\mathfrak{a}}(r) $. (The way to choose may not be unique.) $\hat{c}(r)$ is called the \textbf{min-max generator}.\\
\end{lemma}

\begin{proof}
By a standard compactness argument of Seiberg-Witten equation, $\hat{c}_{r,g_m+e_{\mu}}$ has a convergent subsequence (modulo gauge equivalence)(see \cite{bible}). Just simply choose a limit of such subsequence.\\
\end{proof}

Moreover, we have the following theorem:

\begin{theorem}
$\hat{c}(r)$ satisfies the formula (\ref{buchong}), i.e.,

$$\lim\limits_{r\in U, r\rightarrow \infty} |E(\hat{c}_T(r))-E(\hat{c}(r))| = 0. $$ 

Here, $U, \hat{c}_T (r)$ has the same meaning as in (4) of property \ref{min-max}.\\
\end{theorem}

\begin{proof}
For $r \in U$, the Seiberg-Witten equations for $\hat{c}_T(r)$ and $\hat{c}(r)$ differ by only a small normed $r$-dependent tame perturbation, represented by $p(r)\in P$ (see part (d) of Section 3 in \cite{Weinstein} for details). Moreover, $p(r)$ can be chosen so that $\|p(r)\|_P < \dfrac{1}{\lceil C_r \rceil + 1}$, where $C_r$ is defined in the proof of theorem \ref{independence}, $\lceil C_r \rceil$ is the smallest integer above $C_r$. Thus, fix an $r$, $p(r)$ can play the role of $g_m$ in definition \ref{definition of a(r)} with $m=\lceil C_r \rceil $. Since $\hat{c}_T(r)$ is an actual min-max component of the homology class $\{\sigma\}$ (see \cite{Weinstein}), so by the proof of theorem \ref{independence}, $|\hat{a}_T(r)-\hat{a}(r)| \leq \dfrac{C_r}{\lceil C_r \rceil} \leq 1$.\\

From the inequalities in Section \ref{estimates from taubes} and the definition of actions, it is not hard to see, in any case

$$\dfrac{-2 a}{r} = E+ O(r^{-\frac{1}{3}}E^{\frac{4}{3}}).$$

Since $E(\hat{c}(r))$ is bounded by theorem \ref{main theorem}, so

$$ \lim\limits_{r\in U, r\rightarrow \infty} |E(\hat{c}(r))+\dfrac{2\hat{a}(r)}{r}| = 0.$$

Similarly, $E(\hat{c}_T(r))$ is also bounded (see \cite{Weinstein}), so

$$ \lim\limits_{r\in U, r\rightarrow \infty} |E(\hat{c}_T(r))+\dfrac{2\hat{a}_T(r)}{r}| = 0. $$

Together with $|\hat{a}_T(r)-\hat{a}(r)| \leq \dfrac{C_r}{\lceil C_r \rceil} \leq 1$, one gets

$$\lim\limits_{r\in U, r\rightarrow \infty} |E(\hat{c}_T(r))-E(\hat{c}(r))| = 0. $$ \\

\end{proof}

Notice, the $\hat{c}(r)$ constructed above might not be piecewise continuous when $r>r_k$. The $r$-dependent choices are made in the next section (after choosing generic $\mu$) so that the resulting family (parametrized by r) is piecewise continuous when $r>r_k$.\\

\subsection{Piecewise continuity of $\hat{c}(r)$ for generic $\mu$}\label{piecewise continuity}~\\

Reference \cite{Weinstein} proved that if $\mu$ is generic, then there is a discrete subset in $[2, \infty)$, denoted by $\{p_1, p_2, \cdots\}$, with the following significance: If $r$ is not in this set, then the irreducible solutions of $(SW)_{r,e_{\mu}}$ are distinguished by the values of their actions. Second, for any $i \in \{1, 2,\cdots\}$, the irreducible solutions of $(SW)_{r,e_{\mu}}$ for values of r in the interval $(p_i ,p_{i+1})$ can be identified so as to define continuous and piecewise differentiable families of configurations.\\

Now choose a family $\hat{c}(r)$ in the manner explained previously. This defines the
number $r_k$ as in Definition $\ref{rk}$. If $r > r_k$ and if $r$ is in some interval $(p_i, p_{i+1})$ for $i\in \{1, 2,\cdots\} $, then, because the min-max action $\hat{\mathfrak{a}}$ varies continuously, it follows from the remarks of the preceding paragraph that $\hat{c}(r)$ will vary continuously and piecewise differentiably with $r$ for $r \in (p_i, p_{i+1})$.\\

With the preceding understood, consider next:\\

\begin{lemma}
When $r\in (p_i,p_{i+1})$ and when $r>r_k$, then
$$\dfrac{d\hat{\mathfrak{a}}(r)}{dr} = -\dfrac{1}{2}E(\hat{c}).$$~\\
\end{lemma}

\begin{proof}
This is just because $\hat{c}(r)$ are continuous solutions, thus 
$$\dfrac{d\hat{\mathfrak{a}}(r)}{dr}=(\dfrac{d}{dr}\mathfrak{a}_{r,e_{\mu}})(\hat{c}(r))+<\nabla \mathfrak{a}_{r,e_{\mu}},\dfrac{d}{dr}\hat{c}(r)>= -\dfrac{1}{2}E(\hat{c}(r)).$$~\\
\end{proof}

The proof is almost done. Only the property (3) of  theorem \ref{min-max} needs to be checked. But this property is just a corollary of (4) in definition \ref{done}.\\

~\\

~\\

~\\

~\\

\section{Acknowledgement}~\\

This paper owes much to Professor Cliff Taubes who introduced to me the paper by
Cristofaro-Gardiner, Hutchings and Ramos \cite{huchings} and subsequently gave me much
valuable advice. I am also indebted to Professor Peter Kronheimer and to Boyu
Zhang for helping me with some details about Seiberg-Witten theory. Personally, I am also thankful to the book \cite{bible}, which was a convenient, super encyclopedic reference for almost everything about Seiberg-Witten Floer homology, and Donghao Wang, who first introduced the book \cite{bible} to me. \\

~\\

~\\

~\\

~\\

~\\

~\\

~\\

\bibliography{123}
	\bibliographystyle{plain}

\end{document}